\documentclass{amsart}

\usepackage{enumitem}
\usepackage{amssymb}
\usepackage{latexsym}
\usepackage{longtable}
\usepackage{graphics}
\usepackage{multirow}
\usepackage{array}
\usepackage{tabularx}
\usepackage{amsmath}
\usepackage{geometry}
\usepackage{xcolor}
\usepackage{newtxtext}
\usepackage{newtxmath}
\allowdisplaybreaks[1]
\newtheorem{theorem}{Theorem}[section]
\newtheorem{lemma}[theorem]{Lemma}

\newtheorem{corollary}[theorem]{Corollary}
\theoremstyle{definition}

\newtheorem{example}[theorem]{Example}

\theoremstyle{remark}
\newtheorem{remark}[theorem]{Remark}
\theoremstyle{definition} 
\begin{document}

\title[Subindices and subfactors of $\mathbb{Z}_n$]
{Subindices and subfactors of $\mathbb{Z}_n$ and $k$-index stability of finite groups}

\author{M.H. Hooshmand}

\address{Department of Mathematics, Shi.C., Islamic Azad University, Shiraz, Iran}

\email{MH.Hooshmand@iau.ac.ir , hadi.hooshmand@gmail.com}

%\subjclass[2010]{Primary 11B30, 11P70; Secondary 20D60, 05E15, 20K01, 05C69}

\date{}

\begin{abstract}
We study subindices, subfactors, and index stability in the cyclic group $\mathbb{Z}_n$.
We prove several theorems that not only confirm a conjecture and resolve some open problems about index stability of such groups, but also provide
basic tools for the characterization of finite $k$-index stable groups.
As a consequence, we completely characterize all 2-element index stable subsets of $\mathbb{Z}_n$, obtain an exact closed formula for their density, and
determine all $n$ for which every 2-subset is index unstable. Finally, we present some problems and
a research project extending the study to 3-subsets and general $k$-subsets.
\end{abstract}

\maketitle

%\enlargethispage*{4mm}
%\thispagestyle{empty}

\section{Introduction and preliminaries} \label{Introduction}
In group theory and combinatorics, packing and covering numbers are fundamental tools for analyzing the structure of subsets of groups and their behavior under addition. Packing numbers measure the largest family of pairwise disjoint translates of a set, while covering numbers measure the minimum number of translates required to cover a given set. These notions have a natural geometric interpretation in terms of efficient packing and covering, and they play an important role in the study of sumsets, additive growth, and density phenomena
(see \cite{Furstenberg1981, Lyaskovska2007, Protasov2011}). Recently, the notions of subfactors of subsets, subindices, and index stability in arbitrary groups have been introduced and systematically developed (see \cite{Hooshmand2020, Hooshmand2023, HooshmandYousefian, Hooshmand2026arXiv, Kabenyuk}). These concepts have natural connections with additive combinatorics and number theory, particularly through their relationships with packing and covering numbers, independent sets in Cayley graphs, syndetic sets, and factorization problems in groups.
 Without entering into the general theory (to avoid making this paper too long),
we state the necessary concepts and notations for the important special group $\mathbb{Z}_n$ as follows.

Recall that $\mathbb{Z}_n=\{0,1,2,3,4,\ldots,n-1\}$ is the group of
the least nonnegative residues modulo $n$,
%(with the binary operation $+_n$, see \cite{Hooshmand2021}),
while $\overline{\mathbb{Z}}_n=\{ \overline{0}, \overline{1}, \ldots , \overline{n-1}\} =\mathbb{Z}/n\mathbb{Z}$ denotes the quotient
 group of integers modulo $n$; these two groups are isomorphic, and we use $\mathbb{Z}_n$ throughout.

In order to make it easier to work with the group $\mathbb{Z}_n$, we need to recall the concept of $b$-parts of numbers introduced and studied in \cite{Hooshmand2013}.
Let $b\neq 0$ be a fixed nonzero integer. For all integers $a$ set $$\lfloor a\rfloor_b=b\left\lfloor\frac{a}{b}\right\rfloor \; , \;
 \lceil a\rceil_b=b\left\lceil\frac{a}{b}\right\rceil \; , \; (a)_b=a-\lfloor a\rfloor_b.$$
The $b$-parts of numbers have number-theoretic, algebraic, and analytic aspects (see \cite{Hooshmand2013},
note that $\lfloor a\rfloor_b$ is denoted by $[a]_b$ in it).
For example, if $b$ is a positive integer, then $(a)_b$ is the
remainder of the division of $a$ by $b$. Hence, the binary operation of $\mathbb{Z}_n$ is $+_n$, defined by $x +_n y=(x+y)_n$ (of course, $x +_n y$ is either $x+y$ or $x+y-n$ for all $x,y\in \mathbb{Z}_n$). Note that the additive inverse of $x\neq 0$ in $\mathbb{Z}_n$ is $-_nx=n-x$ (of course, $-_n0=0$), and the difference of two elements is
$$x -_n y=x +_n (-_n y)=(x-y)_n=(x -_n y)_n,$$ for all $x,y\in \mathbb{Z}_n$ (see \cite{Hooshmand2021}). Therefore, in this paper, we work with the additive groups $(\mathbb{Z}_n,+_n,0,-_n)$ and $(\mathbb{Z},+,0,-)$ such that the ground set of $\mathbb{Z}_n$ is a subset of the set of integers $\mathbb{Z}$
(it is worth noting that we have also the grouplike $(\mathbb{Z},+_n)$ with the identity-like 0, see \cite{Hooshmand2021}).

Let $A$ be a non-empty subset of $\mathbb{Z}_n$. A subset $B\subseteq \mathbb{Z}_n$ is a subfactor of $\mathbb{Z}_n$ relative to $A$
if and only if
\begin{equation}\label{subf}
\begin{aligned}
(A -_n A)\cap (B -_n B)&=\{0\},\\
(A -_n A) +_n B&=\mathbb{Z}_n.
\end{aligned}
\end{equation}
It is worth noting that the first condition of  \eqref{subf} holds if and only if the sum $A +_n B$ is direct.
Note that a necessary condition for it is  $(A - A)\cap (B - B)=\{0\}$ as subsets of
$\mathbb{Z}$.
Hence, \eqref{subf} consists of a directness and a covering condition. Also, $B$ is a subfactor of $\mathbb{Z}_n$ relative to $A$
if and only if $(A  +_n b)\cap (A +_n \beta)=\emptyset$, for all distinct elements $b,\beta\in B$,
and $B$ is maximal with this property.

By $|\mathbb{Z}_n:A|^-$ (resp. $|\mathbb{Z}_n:A|^+$) we mean the minimum (resp. maximum)
of values of $|B|$ where $B$ satisfies $(1)$. Also, $A$ is called index stable (in $\mathbb{Z}_n$)
if $|\mathbb{Z}_n:A|^-=|\mathbb{Z}_n:A|^+$, and if it is the case, then the common value is denoted by $|\mathbb{Z}_n:A|$.
For a subgroup $H\leq\mathbb Z_n$, this notation agrees with the usual
subgroup index. The following classical bounds hold for every nonempty subset $A$  (see \cite{Hooshmand2023}):
\begin{equation}\label{bound}
\left\lceil\frac{n}{|A -_n A|}\right\rceil\le |\mathbb{Z}_n:A|^-\le |\mathbb{Z}_n:A|^+\le \left\lfloor \frac{n}{|A|}\right\rfloor
\end{equation}
It is worth noting that all quantities in $(2)$ are equal to 1 if $|A|>\frac{n}{2}$ (in this case we have $A -_n A=\mathbb{Z}_n$).
Also, we have the following basic properties:
\begin{enumerate}[label=(\alph*)]
\item If \(A\subseteq A'\subseteq \mathbb{Z}_n\), then \(|\mathbb{Z}_n:A|^+\geq|\mathbb{Z}_n:A'|^+\). \\
The analogous inequality for the lower index is false in general: for example, in
$\mathbb{Z}_{12}$ with $A=\{0,3,10\}\subseteq A'=\{0,3,9,10\}$ one has $2=|\mathbb{Z}_{12}:A|^-< |\mathbb{Z}_{12}:A'|^-=3$.

\item $|\mathbb{Z}_n:A|^+\ge |\mathbb{Z}_n:A -_n A|^+$ (the analogous inequality for the lower index
does not hold in general).
\item If $A -_n A$ is a subgroup, then $A$ is index stable and $|\mathbb{Z}_n:A|=|\mathbb{Z}_n:A -_n A|$.

\item There are the following important relations among the subindices of a subset $A$, the packing number of $A$, and
the covering number of $A -_n A$ (see \cite{Hooshmand2026arXiv}):
$$
\mathrm{cov}_{\mathbb{Z}_n}(A -_n A)\leq |\mathbb{Z}_n : A|^-\leq |\mathbb{Z}_n : A|^+=\mathrm{pack}_{\mathbb{Z}_n}(A).
$$
\item If $|A|>\frac{n}{3}$, then $A$ is index stable with indices 1 or 2.
\item If $(A -_n A)\cap (B -_n B)=\{0\}$ and $|B|=\lfloor \frac{n}{|A|}\rfloor$, then
($B$ is a subfactor relative to $A$ and so) $(A -_n A) +_n B=\mathbb{Z}_n$.
\end{enumerate}
We first determine the subindices of intervals
$A=\{0,\ldots,m\}\subseteq\mathbb Z_n$.
Since $|A|=m+1$ and $|A -_n A|=\min\{n,2m+1\}$, $(2)$ implies $\{0,\ldots ,m\}$ is index stable if equality $(3)$ holds
(we will show that the converse is also true in Theorem~\ref{mainthm}).
Thus, we arrive at a ceiling-floor equation (which is needed later) in the next lemma.
\begin{lemma}
For a fixed positive integer $m$, the set of all solutions of the equation
\begin{equation}
\left\lceil\frac{n}{2m+1}\right\rceil=\left\lfloor\frac{n}{m+1}\right\rfloor
\end{equation}
is
$$[m+1,3m+2]_\mathbb{Z}\cup\{4m+3\},$$
where $[a,b]_\mathbb{Z}$ denotes $[a,b]\cap \mathbb{Z}$.\\
Moreover
$$
\left\{x\in \mathbb{Z}_+: x=\left\lceil\frac{n}{2m+1}\right\rceil=\left\lfloor\frac{n}{m+1}\right\rfloor, \text{ for some } m,n\in\mathbb{Z}_+ \right\} =\{1,2,3\}.
$$
\end{lemma}
\begin{proof}
First, considering the three cases $m+1\leq n\leq 2m+1$, $2m+2\leq n\leq 3m+2$, and $n=4m+3$,
one can check that such $n$ satisfy the equation and the outputs of the related brackets are $1,2$, and $3$ respectively.

Conversely, suppose that $n$ satisfies the equation, and write
\begin{equation*}
n=(m+1)q+r=(2m+1)q'+r',
\end{equation*}
where $0\leq r\leq m$ and $0\leq r'\leq 2m$. Then
\begin{equation*}
\left\lfloor\frac{n}{m+1}\right\rfloor=q,
\quad
\left\lceil\frac{n}{2m+1}\right\rceil=
\begin{cases}
q', & r'=0,\\
q'+1, & r'>0.
\end{cases}
\end{equation*}
Hence $q=q'$ if $r'=0$, and $q=q'+1$ if $r'>0$.

Moreover, the argument also works for $m=1$ and itself gives
$n\in \{2,3,4,5,7\}$. Thus, we have only the following cases (for $m>1$):

\begin{enumerate}[label=(\alph*)]
\item $r'=0$. We have
$$
\frac{n}{m+1}-1<q=q'=\frac{n}{2m+1}\Rightarrow nm<(m+1)(2m+1)\Rightarrow q'<1+\frac{1}{m}\leq 2,
$$
thus $q'=1$ ($q'> 0$ since $n> 0$) and so
 $n=2m+1$.
\item $r'\neq 0$. We conclude that
$$
n=(m+1)(q'+1)+r=(2m+1)q'+r'\Rightarrow m(q'-1)=r-r'+1.
$$
Now, if $q'=1$, then $r'=r+1$ and so $$2m+2\leq n=2m+2+r\leq 3m+2.$$
But if $q'\neq 1$, then $m\mid r'-r-1$
where $-m\leq r'-r-1\leq 2m-1$. Thus, we have only the following cases since $m>1$ :
\begin{enumerate}[label=(\roman*)]
\item $r'=r-m+1$. Since $1\leq r-m+1\leq 1$, $r=m$ and substituting it in the above equation we obtain
$q=3$ and so $n=4m+3$.
\item $r'=r+m+1$. We conclude that $q=1$ and so $$m+1\leq n=m+1+r=r'\leq 2m.$$
\end{enumerate}
\end{enumerate}
\end{proof}
\begin{remark}
A necessary condition for $n$ to be a solution of the equation $(3)$ is $n\geq m+1$ and so $m\in \mathbb{Z}_n$.
Also, for a fixed $m$ not only the solution set of $(3)$ but also the solution set of
the inequality $\left\lceil\frac{n}{2m+1}\right\rceil\geq \left\lfloor\frac{n}{m+1}\right\rfloor$ is finite
with the upper bound $\frac{(2m+1)(2m+2)}{m}$.
To this end, it should be helpful to consider the signs of the values of $\left\lfloor\frac{n}{m+1}\right\rfloor - \left\lceil\frac{n}{2m+1}\right\rceil$ on the consecutive intervals
\[
\left[1,\, m+1\right), \quad \left[m+1,\, 3m+2\right), \quad \left[3m+2,\, 4m+3\right), \quad \left[4m+3,\, \frac{(2m+1)(2m+2)}{m}\right).
\]
\end{remark}

\section{More results for subindices and index stability in $\mathbb{Z}_n$}
The following theorem evaluates the exact values of the minimum and maximum sizes of the subfactors of $\mathbb{Z}_n$ relative
to $\{0,\ldots ,m\}$. Thus, we can measure the least and the most numbers of maximal disjoint translates
of the set  $\{0,\ldots ,m\}$ in $\mathbb{Z}_n$, by using it.
Moreover, it generalizes Theorem 4.1 of \cite{HooshmandYousefian} and Corollary 3.18 of \cite{Hooshmand2020} (see also Question VII of \cite{Hooshmand2023}) and provides a basic tool for the characterization of finite $k$-index stable groups, which are mentioned as an important unsolved problem in the subindex and subfactor theory of finite groups (cf. Conjecture 4.2 and the New Problem of \cite{HooshmandYousefian}).

\begin{theorem}\label{mainthm}
For every $m\in \mathbb{Z}_n$ we have
\begin{equation}
|\mathbb{Z}_n:\{0,\ldots ,m\}|^-=\left\lceil\frac{n}{2m+1}\right\rceil \; , \; |\mathbb{Z}_n:\{0,\ldots,m\}|^+=\left\lfloor\frac{n}{m+1}\right\rfloor
\end{equation}
\end{theorem}
\begin{proof}
We put $G=\mathbb{Z}_n$, $A=\{0,\ldots ,m\}$, and follow the next steps.\\
\textbf{Step 1.}  If $m> \frac{n-2}{2}$
, then $1\leq \frac{n}{m+1}<2$, $0<\frac{n}{2n-1}\leq \frac{n}{2m+1}\leq 1$,  $|A|>\frac{n}{2}$, and
so
$$
|G:A|^-=\left\lceil\frac{n}{2m+1}\right\rceil=|G:A|^+=\left\lfloor\frac{n}{m+1}\right\rfloor=1.
$$
Also, if  $\frac{n-3}{3}< m\leq \frac{n-2}{2}$
, then
$$
|G:A|^-=\left\lceil\frac{n}{2m+1}\right\rceil=|G:A|^+=\left\lfloor\frac{n}{m+1}\right\rfloor=2,
$$
by (2).\\
Hence, for the next steps we can assume that $1\leq m\leq \frac{n-3}{3}$ and so $n\geq 3m+3\geq 6$
 (if $m=0$, then the equalities clearly hold and all are equal to $n$).\\
\textbf{Step 2.} We have
$$
A -_n A=\{0,1,\ldots ,m, n-m,n-m+1,\ldots , n-1\}\neq G\; , \; \mathcal{C}^0(A):=(\mathbb{Z}_n\setminus(A -_n A))\cup\{0\}=\{0,m+1,\ldots, n-m-1\}
$$
Put $t=\left\lfloor \frac{n}{m+1}\right\rfloor$ and
 $B=\left\{0,m+1,2(m+1),\ldots,(t - 1)(m+1)\right\}$.
 Write
\[
n=(2m+1)q+s,\qquad 0\leq s\leq 2m.
\]
Thus $s=(n)_{2m+1}$ (the remainder of the division of $n$ by $2m+1$), and $q=\left\lfloor\frac{n}{2m+1}\right\rfloor=\left\lceil\frac{n}{2m+1}\right\rceil-1$ when $s\neq 0$.

Define
\[
B'=
\begin{cases}
\{j(2m+1):0\leq j\leq q-1\}, & s=0;\\[1mm]
\{j(2m+1):0\leq j\leq q\}, & m+1\leq s\leq 2m;\\[1mm]
\{j(2m+1):0\leq j\leq q-1\}\cup\{q(2m+1)-m\}, & 1\leq s\leq m.
\end{cases}
\]
Putting $k=q+1-j$, we observe that the general form of the elements of $B'$ in the second and third cases, except for the last elements, is
$$
\left(q+1- k\right)(2m+1)=\lceil n\rceil_{2m+1}-k(2m+1)\; ; \; 2\leq k\leq q+1,
$$
and the last element of the second (resp. third) case is equal to $\lceil n\rceil_{2m+1}-2m-1$ (resp. $\lceil n\rceil_{2m+1}-3m-1$).
A similar general form can also be written for $B$.
We have $|B|=\lfloor \frac{n}{m+1}\rfloor$, $|B'|=\lceil \frac{n}{2m+1}\rceil$, and $B,B'$ are (well-defined and) subsets of
$\mathcal{C}^0(A)$. Because $m\leq \frac{n-2}{2}$ if and only if $$m+1\leq \lfloor n\rfloor_{m+1}-m-1\leq n-m-1.$$ Also,
for $B'$, note that $s\geq m+1$ if and only if $\lfloor n\rfloor_{2m+1}\leq n-m-1$,
and $\lfloor n\rfloor_{2m+1}=\lceil n\rceil_{2m+1}-(2m+1)$ since $2m+1$ does not divide $n$ (equivalently  $(n)_{2m+1}\neq 0$).
If $s=0$, then
$$B'=\left\{0,2m+1,\ldots,n-2m-1\right\},$$
that is clearly a subset of $\mathcal{C}^0(A)$ and also a subgroup of $G$.\\
Now, if $1\leq s\leq m$, then
$$
\frac{n-1}{2m+1}\geq \frac{n-s}{2m+1}=q
$$
and so
$$
\lceil n\rceil_{2m+1}-3m-1\leq (n-1+2m+1)-3m-1=n-m-1.
$$
Also, note that $q\geq 2$ for if not, then $$n=q(2m+1)+s\leq 3m+1.$$
Therefore,  $\lceil n\rceil_{2m+1}-4m-2\geq 2m+1$ (we need this for the definition of $B'$ in the second case).\\
  \textbf{Step 3.} In this step we prove that $B-_nB\subseteq \mathcal{C}^0(A)$ and $B'-_nB'\subseteq \mathcal{C}^0(A)$.
The general form of the elements of $B-_nB$ is $\kappa (m+1)$ or $n+\kappa (m+1)$ where $1-t\leq \kappa \leq t-1$.
Such elements belong to $\mathcal{C}^0(A)$ if $0\leq \kappa\leq t-1$, obviously. Hence, let
$1-t\leq \kappa\leq -1$, then
$$
m+1\leq n-(t-1)(m+1)\leq \kappa (m+1)+n\leq n-m-1
$$
Therefore $B -_n B\subseteq \mathcal{C}^0(A)$.\\
 Note that the general form of elements of $B'-_n B'$ in the second (resp. first) case is $\kappa (2m+1)$ or $n+\kappa (2m+1)$, where $-q\leq \kappa \leq q$ (resp. $1-q\leq \kappa \leq q-1$).
%  $(\kappa-1) (2m+1)-m$ and $n-((\kappa-1) (2m+1)-m)$, where $2\leq\kappa\leq q+1$,  if $1\leq s\leq m$.
 Hence, if $s\geq m+1$ and $-q\leq \kappa \leq -1$, then
 $$
 m+1\leq s=n-\lceil n\rceil_{2m+1}+2m+1\leq n+\kappa (2m+1)\leq n-2m-1\leq n-m-1,
 $$
and so $n+\kappa (2m+1)\in \mathcal{C}^0(A)$ (this also holds for the case that $\kappa$ is non-negative, obviously).
Also, if $s=0$, then $B'-_n B'=B'\subseteq \mathcal{C}^0(A)$. \\
Now, if $1\leq s\leq m$, then we have the following cases for the elements of $B'-_nB'$.
\begin{enumerate}[label=(\roman*)]
\item $\kappa (2m+1)$ where $0\leq \kappa \leq q-1$. Such elements lie in $\mathcal{C}^0(A)$ clearly.
\item $n+\kappa (2m+1)$ where $ 1-q \leq \kappa \leq -1$. We have
$$
2m+1<s+2m+1= n-(\lceil n\rceil_{2m+1}-2(2m+1))\leq n+\kappa (2m+1)\leq n-2m-1\leq n-m-1,
$$
and so such elements lie in $\mathcal{C}^0(A)$.
\item $(\kappa-1) (2m+1)-m$, where $2\leq\kappa\leq q+1$. We conclude that
$$
m+1\leq (\kappa-1) (2m+1)-m\leq \lceil n\rceil_{2m+1}-3m-1\leq n-m-1,
$$
and thus we are done for this case.
\item $n-((\kappa-1) (2m+1)-m)$, where $2\leq\kappa\leq q+1$.
$$
m+1\leq s+m= n-(\lceil n\rceil_{2m+1}-(3m+1))\leq n-((\kappa-1) (2m+1)-m)\leq n-m-1,
$$
and hence this step is complete.
\end{enumerate}
\textbf{Step 4.} Here, we prove that $(A -_n A) +_n B'=G$. If $s=0$, then
it is confirmed since $B'$ is a subgroup and $A -_n A$ is a transversal of $B'$ in $G$.
Hence, either $1\leq s\leq m$ or $m+1\leq s\leq 2m$. Note that every element of $(A-_nA)+_nB'$ is
of the form $\delta+_n b$, where $\delta\in A-_nA$ and $b\in B'$. Fix $x\in G$ and put $r=(x)_{2m+1}$, $x'=\lfloor \frac{x}{2m+1}\rfloor$.
Then, we have $x'\leq q$ (because $\left\lfloor \frac{x}{2m+1} \right\rfloor
\le\left\lfloor \frac{n-1}{2m+1} \right\rfloor=q$) and the following cases.
\begin{enumerate}[label=(\roman*)]
\item $r=0$. If $0\leq x'\leq q-1$, then putting $k=q+1-x'$
one may choose $\delta=0$ and $b=\lceil n\rceil_{2m+1}-k(2m+1)\in B'$ to obtain $x=\delta+b$.  But
if $x'=q$, then $$x=\lceil n\rceil_{2m+1}-(2m+1)+0\in (A -_n A) +_n B'$$ if $s\geq m+1$,
and $$x=\lceil n\rceil_{2m+1}-(3m+1)+m\in (A  -_n A) +_n B'$$ if $s\leq m$.
\item $1\leq r\leq m$. If $0\leq x'\leq q-1$, then it is enough to take $k=q+1-x'$
and $\delta=r$. But if $x'=q$, then choose $k=q+1-x'$ and $\delta=r$ for the case $s\geq m+1$,
and for the case $1\leq s\leq m$, considering $\lceil\frac{x}{2m+1}\rceil=x'+1=q+1$ we have
$$
1\leq n-x=(2m+1)x'+(n)_{2m+1}-((2m+1)x'+(x)_{2m+1})=(n)_{2m+1}-(x)_{2m+1}=s-r\leq m-1,
$$
and so $n-m+1\leq x\leq n-1$ which means $x\in A -_n A\subseteq (A -_n A) +_n B'$.
\item $r\geq m+1$. For this case, we consider the following table values for choosing
$\delta$ and $b$:

\begin{center}
\begin{tabular}{l l l l}
$x'$               & Condition          & $b$            & $\delta$\\
\hline
$0\leq x'\leq q-2$ & $0\leq s\leq 2m$  & $(x'+1)(2m+1)$ & $n+r-(2m+1)$ \\
$x'=q-1$           & $s=0$              & $0$            & $x$ \\
$x'=q-1$           & $1\leq s\leq m$    & $q(2m+1)-m$    & $r-m-1$ \\
$x'=q-1$           & $m+1\leq s\leq 2m$ & $q(2m+1)$      & $n+r-(2m+1)$ \\
$x'=q$             & $r<s$              & $0$            & $x$
\end{tabular}
\end{center}
For details, if $0\leq x'\leq q-2$, then putting $k=q-x'$,
$\delta=r+n-2m-1$ (since $n-m\leq r+n-2m-1\leq n-1$), and $b=(x'+1)(2m+1)$ we obtain
$$\delta+_n b=((2m+1)x'+r+n)_n=(2m+1)x'+r=x.$$
But if $x'=q-1$, then
choose $k=q-x'$ and $\delta=r$ for the case $s\geq m+1$,
and  $k=1$ and $\delta=r-m-1$ if $1\leq s\leq m$. Finally, if $x'=q$, then
$$
x=(2m+1)q+r\geq (2m+1)(\frac{n}{2m+1}-1)+r\geq n-m,
$$
thus $x\in A-_nA\subseteq (A-_nA)+_nB'$.
\end{enumerate}
\textbf{Step 5.} From the steps 3 and 4, and since $|B|=\lfloor \frac{n}{m+1}\rfloor=\lfloor \frac{|G|}{|A|}\rfloor$,
we conclude that $B$ and $B'$ are subfactors of $G$ relative to $A$  (i.e., maximal with respect to this property). Therefore,
$\{0,\ldots ,m\}$ attains the maximum possible upper index and the minimum possible lower index permitted by $(2)$.
\end{proof}
The next corollary is a direct result of the above theorem and Lemma 1.1.
\begin{corollary}\label{cor:stable} Let $1\leq m\leq n-1$.
\begin{enumerate}[label=(\alph*)]
\item  The set $\{0,\ldots ,m\}$ is index stable in $\mathbb{Z}_n$ if and only if either $n\leq 3m+2$ or $n=4m+3$
(equivalently, $m\ge \lfloor\frac{n}{3}\rfloor$ or $m=\frac{n-3}{4}$). Moreover, when this holds the index is
1,2, or 3.
\item $\{0,\ldots ,m\}$ takes its  maximum upper and minimum lower indices. Hence,
it is strong index unstable if and only if $n \geq 3m+3$ and $n\neq 4m+3$.
\end{enumerate}
\end{corollary}

The above theorem gives another lower bound for $|\mathbb{Z}_n:A|^+$ for every subset $A$.
\begin{corollary}\label{cor:span-bound}
For every nonempty subset \(A\subseteq \mathbb{Z}_n\) we have
\begin{equation}\label{eq:span-chain}
|\mathbb{Z}_n:A|^+
\geq
\max_{t\in\mathbb Z_n}
\left\lfloor
\frac{n}{\max(A-_n t)+1}
\right\rfloor
\geq \left\lfloor\frac{n}{\max A-\min A+1}\right\rfloor
\geq \left\lfloor\frac{n}{\max A+1}\right\rfloor.
\end{equation}
Hence, we can improve the lower bound of $|\mathbb{Z}_n:A|^+$ as follows
\begin{equation}\label{best-bound}
|\mathbb{Z}_n:A|^+\;\ge\;\max\left(
\max_{t\in\mathbb Z_n}\left\lfloor\frac{n}{\max(A-_nt)+1}\right\rfloor,\;
\left\lceil\frac{n}{|A-_nA|}\right\rceil
\right).
\end{equation}
\end{corollary}
\begin{proof}
Since $A\subseteq \{0,\cdots, \max A\}$ and $|\mathbb{Z}_n:A|^+=|\mathbb{Z}_n:A-_n t|^+$, for all $t\in\mathbb Z_n$, the above theorem together with the property (a) (Section 1) give the first inequality in \eqref{eq:span-chain}. For the second, take $t=\min A$; since $0\le a-\min A\le \max A-\min A<n$ for every $a\in A$,  $\max(A-_n(\min A))=\max A-\min A$ exactly, and the left side of \eqref{eq:span-chain}, maximizing over all $t$, is at least this value. The combined bound \eqref{best-bound} is obtained by \eqref{bound} and \eqref{eq:span-chain}.
\end{proof}

The following examples show that neither of the two terms in \eqref{best-bound}
(and its weaker relatives $\left\lfloor\frac{n}{\max A-\min A+1}\right\rfloor$,
$\left\lfloor\frac{n}{\max A+1}\right\rfloor$) dominates the other in general.
 We also show that each inequality in
\eqref{eq:span-chain}, and the combined bound \eqref{best-bound} itself, can be strict.

\begin{example}
Let $n>8$, and consider $A=\{1,2,5\}$ as a subset of $\mathbb{Z}_n$. Then $\max A+1=6$,
$\max A-\min A+1=5$, and $|A -_n A|=7$. Then, we have
\[
\begin{array}{c|c|c|c|c|c}
n & |\mathbb{Z}_n:A|^+
& \max_t\left\lfloor\dfrac{n}{\max(A-_nt)+1}\right\rfloor
  & \left\lfloor\dfrac{n}{\max A-\min A+1}\right\rfloor
  & \left\lfloor\dfrac{n}{\max A+1}\right\rfloor
  & \left\lceil\dfrac{n}{|A-_nA|}\right\rceil \\ \hline
9  & 2 & 1 & 1 & 1 & 2 \\
10 & 2 & 2 & 2 & 1 & 2 \\
12 & 3 & 2 & 2 & 2 & 2 \\
15 & 4 & 3 & 3 & 2 & 3 \\
20 & 5 & 4 & 4 & 3 & 3 \\
29 & 8 & 5 & 5 & 4 & 5 \\
42 & 12 & 8 & 8 & 7 & 6
\end{array}
\]

Moreover, for $A=\{0,1,8\}\subseteq\mathbb{Z}_{10}$, taking $t=8$ gives $A-_{10}8=\{0,2,3\}$, so
\[
\max_{t\in\mathbb Z_{10}}\left\lfloor\frac{10}{\max(A-_{10}t)+1}\right\rfloor
\;\ge\;\left\lfloor\frac{10}{3+1}\right\rfloor=2
\;>\;1=\left\lfloor\frac{10}{8-0+1}\right\rfloor=\left\lfloor\frac{n}{\max A-\min A+1}\right\rfloor,
\]
showing that the first inequality in \eqref{eq:span-chain} can be strict.

Finally, the combined bound \eqref{best-bound} is not always tight: for $A=\{0,1,5\}\subseteq\mathbb{Z}_9$,
one has $A-_9A=\{0,1,4,5,8\}$, so
$B=\{0,3,6\}$ with $|B|=\lfloor\frac{9}{|A|}\rfloor$ satisfies $(A-_9A)\cap(B-_9B)=\{0\}$ (see (f) in Section 1), and so
\[
|\mathbb{Z}_9:A|^+=3\;>\;2=\max\left(1,\,2\right)
=\max\left(\max_t\left\lfloor\frac{9}{\max(A-_9t)+1}\right\rfloor,\left\lceil\frac{9}{|A-_9A|}\right\rceil\right).
\]
\end{example}

Given the above corollary and example, the question arises that can all the inequalities of \eqref{eq:span-chain} be strict?

The answer is negative.
Let $A=\{a_1<\cdots<a_k\}$, $m=a_1$, $M=a_k$, $L=M-m$, let $g_i=a_{i+1}-a_i$ be the internal
gaps ($1\le i\le k-1$, so $\sum_i g_i=L$), let $g_k=n-L$ be the wrap gap, and set
$g^*=\max_{1\le i\le k-1}g_i$. A direct case analysis on the position of $t\in\mathbb Z_n$ relative
to $m$ and $M$ shows
\[
\min_{t\in\mathbb Z_n}\max(A-_nt)=n-D,\qquad D:=\max(g^*,n-L),
\]
so the second inequality of \eqref{eq:span-chain} reads
$\left\lfloor\frac{n}{n-D+1}\right\rfloor\ge\left\lfloor\frac{n}{L+1}\right\rfloor$. If this is strict,
then $D>n-L$, forcing $D=g^*$ and $g^*>n-L$; since $g^*\le\sum_ig_i=L$, this gives $n-L<L$, and hence $\left\lfloor\frac{n}{L+1}\right\rfloor\le1$. On the other hand
$M\le n-1$ gives $M+1\le n$, so $\left\lfloor\frac{n}{M+1}\right\rfloor\ge1$. Thus
$\left\lfloor\frac{n}{L+1}\right\rfloor\le1\le\left\lfloor\frac{n}{M+1}\right\rfloor$, so the third
inequality of \eqref{eq:span-chain} cannot also be strict.

Two of the three inequalities \emph{can} be simultaneously strict: the examples
$A=\{30,70\}\subseteq\mathbb{Z}_{100}$ and $A=\{0,1,8\}\subseteq\mathbb{Z}_{10}$ above realize the
two possible combinations.

\begin{remark}[Sketch of an alternative graph-theoretic proof of Theorem~\ref{mainthm}]
Let
$$S:=(A -_n A)\setminus\{0\}= \{1,\ldots ,m, n-m,n-m+1,\ldots , n-1\}\equiv \{\pm1,\ldots,\pm m\}\pmod n,$$
and $\Gamma=\mathrm{Cay}(\mathbb{Z}_n,S)$. For $x\ne y$, writing $d=y -_n x$, we have
$$d\in S \iff \min(d,n-d)\le m,$$
 i.e. $x,y$ have cyclic distance $\le m$; hence $\Gamma=C_n^m$ (that is the circulant graph on \(\mathbb{Z}_n\) with edges between distinct vertices whose cyclic distance is at most \(m\)).

The condition \( (A -_n A)\cap(B -_n B)=\{0\}\) means \(B\) is an independent set in \(\Gamma\), while \((A -_n A) +_n B=\mathbb{Z}_n\) means \(B\) is a dominating set
(see \cite{Kabenyuk}). Hence subfactors relative to \(A\) are precisely the maximal independent sets of \(\Gamma\). Therefore
\[
|\mathbb{Z}_n:A|^+ = \alpha(\Gamma), \qquad |\mathbb{Z}_n:A|^- = i(\Gamma),
\]
where \(\alpha\) is the independence number and \(i\) the independent domination number.

For \(\Gamma=C_n^m\), a greedy selection gives \(\alpha(\Gamma)=\left\lfloor \frac{n}{m+1}\right\rfloor\). For \(i(\Gamma)\), maximality of an independent set forces consecutive gaps \(d_j\) to satisfy \(m+1\le d_j\le 2m+1\); with \(k\) vertices, this gives
\[
k(m+1)\le n\le k(2m+1),
\]
so \(k\ge \left\lceil \frac{n}{2m+1}\right\rceil\). The bound is attainable by choosing gaps summing to \(n\), hence
\[
i(\Gamma)=\left\lceil \frac{n}{2m+1}\right\rceil.
\]
Substitution yields exactly (4).
\end{remark}
Motivated by the above remark, we observe that subfactors exist for every cardinality mentioned in the following theorem, proposed by Mikhail Kabenyuk.
\begin{theorem}
For every $m\in \mathbb{Z}_n$, the set of the sizes of all subfactors of $\mathbb{Z}_n$ relative to
$\{0,\ldots ,m\}$ is equal to
\begin{equation*}
\left[
\left\lceil\frac{n}{2m+1}\right\rceil,
\left\lfloor\frac{n}{m+1}\right\rfloor
\right]_{\mathbb Z}.
\end{equation*}
\end{theorem}
\begin{proof}
 For a given $k$, put
\begin{equation*}
r=n-k(m+1).
\end{equation*}
If
\begin{equation*}
\left\lceil\frac{n}{2m+1}\right\rceil
\leq k\leq
\left\lfloor\frac{n}{m+1}\right\rfloor,
\end{equation*}
then
\begin{equation*}
0\leq r\leq km.
\end{equation*}
Therefore, one can choose $e_1,\ldots,e_k\in\{0,\ldots,m\}$ such that
\begin{equation*}
e_1+\cdots+e_k=r,
\end{equation*}
and put
\begin{equation*}
d_i=m+1+e_i.
\end{equation*}
Then
\begin{equation*}
m+1\leq d_i\leq2m+1,
\qquad
d_1+\cdots+d_k=n.
\end{equation*}
Choosing $k$ points on the cycle with successive gaps
$d_1,\ldots,d_k$, we obtain a subfactor of cardinality $k$.
\end{proof}

A subset $A$ of $\mathbb{Z}_n$ is
 strong index unstable (see \cite{HooshmandYousefian}) if and only if
$$\left\lceil\frac{n}{|A -_n A|}\right\rceil=|\mathbb{Z}_n:A|^-<|\mathbb{Z}_n:A|^+= \left\lfloor \frac{n}{|A|}\right\rfloor.$$
For example, consider $A=\{0,1,\dots,m\}$ in $\mathbb{Z}_{n}=\mathbb{Z}_{3m+3}$ which gives $|\mathbb{Z}_n:A|^-=2$, $|\mathbb{Z}_n:A|^+=3$ (e.g., $m=2$, $n=9$). \\

\textbf{Problem 1.} Characterize or classify all subsets of $\mathbb{Z}_n$ that are strong index unstable.
\\

\textbf{Note.}
\begin{itemize}
\item  Regarding the above problem, there are some related interesting properties. As we know, by (\ref{bound}), the set of all values $|B|$ for which $B$ satisfies (\ref{subf}) is contained in
\begin{equation*}
\left[
\left\lceil\frac{n}{|A -_n A|}\right\rceil,
\left\lfloor\frac{n}{|A|}\right\rfloor
\right]_{\mathbb Z},
\qquad A\neq\emptyset.
\end{equation*}

\textbf{Problem 2.} Characterize or classify all subsets of $\mathbb{Z}_n$ for which this set of values equals the full interval, exactly contains two of its endpoints, or at least contains two of its endpoints.

For example, all subsets $A=\{0,\ldots ,m\}$ enjoys the full interval by the above theorem.
Such subsets $A$ are moreover strong index unstable whenever $\left\lceil\frac{n}{|A -_n A|}\right\rceil\ne\left\lfloor\frac{n}{|A|}\right\rfloor$.
\item
One may also consider the absolute strong unstability in $\mathbb{Z}_n$, that is $2 = |\mathbb{Z}_n:A|^- < |\mathbb{Z}_n:A|^+ = \lfloor n/2 \rfloor$.
For example:  $A=\{0,1\}$ in $\mathbb{Z}_6$ gives $|\mathbb{Z}_6:A|^-=2$, $|\mathbb{Z}_6:A|^+=\lfloor6/2\rfloor=3$.
Indeed, Lemma 3.1 in the next section shows that only $\mathbb{Z}_6$ contains such subsets.
\end{itemize}

There is an important basic conjecture in the theory of subindices of finite groups as follows:\\
\cite[Conjecture 4.2]{HooshmandYousefian}\textbf{.} If $n > 11$, then $\mathbb{Z}_n$ is not $k$-index stable if and only if $2\leq k \leq \lfloor \frac{n}{3} \rfloor$.\\
The group $\mathbb{Z}_n$ is $k$-index stable if and only if all subsets of the size $k$ are index stable.
For example, it is 1-index stable, and it is $k$-index stable for all  $\frac{n}{3}<k\leq n$.
The above conjecture  is an essential tool for
characterization or classification of finite $k$-index stable groups. Now we can prove it by using the main
theorem.
\begin{theorem}
If $n>11$, then $\mathbb{Z}_n$ is not $k$-index stable if and only if $2\leq k \leq \lfloor \frac{n}{3} \rfloor$.
\end{theorem}
\begin{proof}
 If $\mathbb{Z}_n$ is not $k$-index stable, then $2\leq k \leq \lfloor \frac{n}{3} \rfloor$ (as mentioned before).
 Conversely, if $2\leq k \leq \lfloor \frac{n}{3} \rfloor$, then put $m=k-1$.
Now, consider the following two cases.
\begin{enumerate}[label=Case \arabic*.]
\item $n\neq 4m+3$. Since $4m+3\neq n>3m+2$, Theorem~\ref{mainthm} and Corollary~\ref{cor:stable} imply $\{0,\ldots ,m\}$ of size $k$ is not index stable.
\item $n=4m+3$. In this case $\mathbb{Z}_n=\mathbb{Z}_{4k-1}$ and $k\geq 4$, $n\geq 15$. We claim that  the set
$\{0,\ldots ,m-1,m+1\}=\{0,\ldots ,k-2,k\}$ with the size $k$ is not index stable. Indeed, we have
$$
A  -_n A=\{0,\ldots ,k,3k-1,3k,\ldots,4k-2\}\; , \; \mathcal{C}^0(A)=\{0,k+1,k+2,\ldots,3k-2\}.
$$
Thus $|A  -_n A|=2k+1$ and
$$
2=\left\lceil \frac{4k-1}{2k+1}\right\rceil\leq |B|\leq \left\lfloor \frac{4k-1}{k}\right\rfloor=3,
$$
for all subfactors $B$ of $\mathbb{Z}_n$ relative to $A$. On the other hand $B_1=\{0,k+1,2k+2\}$ and $B_2=\{0,2k-2\}$ are subfactors of $\mathbb{Z}_n$ relative to $A$.
Because
\[
B_1 -_n B_1=\{0,k+1,2k+2,2k-3,3k-2\},
\]
and
\[
B_2 -_n B_2=\{0,2k-2,2k+1\}.
\]

Since
\[
(A  -_n A)\cap(B_1 -_n B_1)=\{0\},
\;
(A  -_n A)\cap(B_2 -_n B_2)=\{0\},
\]
\(A +_n B_1\) and \(A +_n B_2\) are direct.

Moreover,
\[
(A  -_n A) +_n B_2
=
\{0,\ldots,2k\}
\cup
\{2k-2,\ldots,3k-2\}
\cup
\{3k-1,\ldots,4k-2\}
=
\mathbb Z_{4k-1}.
\]
Hence \(B_1\) and \(B_2\) are subfactors of \(\mathbb Z_{4k-1}\) relative to \(A\) (note that $|B_1|$ takes the maximum value).
Now, due to the above inequality, and since \(|B_2|=2\) and \(|B_1|=3\), we obtain
\[
2\le |\mathbb Z_{4k-1}:A|^-\le 2<3\le |\mathbb Z_{4k-1}:A|^+\le 3.
\]
Therefore
\[
|\mathbb Z_{4k-1}:A|^-=2,
\qquad
|\mathbb Z_{4k-1}:A|^+=3.
\]
Thus \(A\) is index unstable.
\end{enumerate}
\end{proof}
About the above theorem, indeed, if $6\leq n\neq 7, 11$, then $\mathbb{Z}_n$ is not $k$-index stable if and only if $2\leq k \leq \lfloor \frac{n}{3} \rfloor$.
One can check it for the case $n\leq 11$ by Table 1 of \cite{HooshmandYousefian}.\\
As another result of the above theorem, we have  $\mathbb{Z}_{4k-1}$ is not $k$-index stable if $k\geq 4$.
\begin{corollary}
If $k\geq 4$, then $|\mathbb{Z}_{4k-1}:\{0,\ldots ,k-2,k\}|^-=2$, $|\mathbb{Z}_{4k-1}:\{0,\ldots ,k-2,k\}|^+=3$
and so $\mathbb{Z}_{4k-1}$ is not $k$-index stable.
\end{corollary}

\subsection{The density of index stable and index unstable $2$-subsets of $\mathbb{Z}_n$}

For studying the index stable $2$-subsets of $\mathbb{Z}_n$, first note that
for $A=\{a,b\}\subseteq \mathbb{Z}_n$, we have
\[
|\mathbb{Z}_n:\{a,b\}|^\pm
=
|\mathbb{Z}_n:\{0,(a-b)_n\}|^\pm=|\mathbb{Z}_n:\{0,(b-a)_n\}|^\pm .
\]
Moreover, if $\gcd(a-b,n)=1$ (equivalently  $\gcd((a-b)_n,n)=1$) , then $(a-b)_n=a-_n b$ is invertible in the ring $(\mathbb{Z}_n,+_n,\cdot_n)$ (where the multiplication is defined by $x\cdot_n y:=(xy)_n=xy-\lfloor xy\rfloor_n$), and multiplying $\{0,a-_n b\}$ by $(a-_n b)^{-1}$ (see Remark 3.17 of \cite{Hooshmand2026arXiv}), we arrive at
\[
|\mathbb{Z}_n:\{0,a-_n b\}|^\pm
=
|\mathbb{Z}_n:\{0,1\}|^\pm .
\]
The converse of this implication is false, for example, in $\mathbb{Z}_{12}$ one can check that $|\mathbb{Z}_{12}:\{0,2\}|^\pm = |\mathbb{Z}_{12}:\{0,1\}|^\pm$ while $\gcd(2,12)=2>1$. Thus the condition $\gcd(a-b,n)=1$ is sufficient for equality of subindices but not necessary (every $2$-element subset whose difference is a unit has the same subindices as $\{0,1\}$).

Hence, for the general case, we do it in several steps as follows.

\subsubsection*{Step 1.}

We first claim that the set of all $2$-element subsets of $\mathbb{Z}_n$, denoted by $P_2(\mathbb{Z}_n)$, is equal to
$$
P_2(\mathbb{Z}_n)=\begin{cases} \left\{x+\{0,d\}  \text{: } 1\leq d\leq \frac{n-1}{2}\; ,\; 0\leq x\leq n-1\right\} & \text{ if } n \text{ is odd } \\
\left\{x+\{0,d\}  \text{: } 1\leq d\leq \frac{n}{2}-1\; ,\; 0\leq x\leq n-1\right\}\dot{\cup} \left\{x+\{0,\frac{n}{2}\}  \text{: }  0\leq x\leq \frac{n}{2}-1\right\} & \text{ if } n \text{ is even }
\end{cases}
$$
where $\dot{\cup}$ denotes the disjoint union, and every $2$-element subset admits a unique representation of the stated form.
\begin{proof}
Let $\{a,b\} \in P_2(\mathbb{Z}_n)$. Setting $x = b$ and $d = (a - b)_n$, we have
\[
\{a,b\} = b +_n \{0, (a-b)_n\} = x +_n \{0,d\}= \{x,x+_n d\},
\]
where $d \in \{1,\ldots,n-1\}$. Note the shift identity
\[
x +_n\{0,d\} \;=\; (x +_n d) +_n \{0,\,n - d\},
\]
since both sides equal $\{x,x+_n d\}$ as subsets of $\mathbb Z_n$. Thus the \emph{same} subset can be re-parametrized by replacing $(x,d)$ with $(x+_n d,\,n-d)$; the base point changes to $x+_n d$ when we pass from gap $d$ to gap $n-d$. This lets us normalize the gap into a restricted range.
\begin{enumerate}[label=Case \arabic*.]
\item $n$ is odd.
For any $d \in \{1,\ldots,n-1\}$ we have $d \neq n-d$ (since $n$ is odd, $2d \not\equiv 0$). If $d > \frac{n-1}{2}$, apply the shift identity to replace $(x,d)$ by $(x+_n d,\,n-d)$; now $n-d$ satisfies $1 \le n-d \le \frac{n-1}{2}$. So every subset admits a representation with $1 \leq d \leq \frac{n-1}{2}$ and $x\in\{0,\ldots,n-1\}$.

\emph{Uniqueness.} Suppose $x+_n\{0,d\} = x'+_n\{0,d'\}$ with $1\leq d,d'\leq \frac{n-1}{2}$ and $0\leq x,x'\leq n-1$. Then $\{x, x+_n d\} = \{x', x'+_n d'\}$ in $\mathbb{Z}_n$. Either $x=x'$ and $x+_n d=x'+_n d'$, giving $d=d'$ and $x=x'$; or $x = x'+_n d'$ and $x+_n d = x'$, giving $d'=-_n d=n-d$. But $2\leq d+d'\leq n-1$, a contradiction. Hence the representation is unique.
\item $n$ is even.
Here $d = \frac{n}{2}$ satisfies $d = n-d$, i.e.\ the shift identity becomes $$x+_n\{0,\tfrac n2\}=(x+_n\tfrac n2)+_n\{0,\tfrac n2\}:$$
 a fixed point of the re-parametrization. Thus each such subset has exactly two representations $x+_n\{0,\frac{n}{2}\}$ and $(x+_n\frac{n}{2})+_n\{0,\frac{n}{2}\}$, so we normalize by requiring $0\leq x \leq \frac{n}{2}-1$.
For $d \in \{1,\ldots,n-1\}\setminus\{\frac{n}{2}\}$ we have $d \neq n-d$, and as in Case~1 we apply the shift identity when needed to choose $1\leq d\leq \frac{n}{2}-1$, with $x \in \{0,\ldots,n-1\}$.

Note that the two families are disjoint, since in the second family their two elements differ by exactly $\frac{n}{2}$, whereas in the first family the difference $d$ satisfies $1\leq d\leq \frac{n}{2}-1$.

\emph{Uniqueness.}  For the first family follows by the same argument as in Case~1.
For uniqueness within the second family, if $x+_n\{0,\tfrac n2\}=x'+_n\{0,\tfrac n2\}$ with $0\le x,x'\le \tfrac n2-1$, then either $x=x'$, or $x-_n x'=\tfrac n2$; the latter is impossible since $x-x'=(x-x')_n+kn$, for some integer $k$, and $$x-x'\in\{-(\tfrac n2-1),\ldots,\tfrac n2-1\}.$$
Hence $x=x'$, and the representation is unique.
\end{enumerate}
\end{proof}

\subsubsection*{Step 2.}
Putting $d=(a-b)_n$ and $g=\gcd(n,d)$, we claim that
\begin{equation}\label{eq:2subset-formula}
|\mathbb{Z}_n : \{a, b\} |^+ = g \left\lfloor \frac{n}{2g} \right\rfloor= \left\lfloor \frac{n}{2} \right\rfloor_{g}, \;
|\mathbb{Z}_n : \{a, b\} |^- = g\left\lceil \frac{n}{3g} \right\rceil= \left\lceil \frac{n}{3} \right\rceil_{g}.
\end{equation}
\begin{proof}
Let \(m = n/g\), and \(H = \langle d \rangle \leq \mathbb{Z}_n\). Since \(d\) has order \(m\) in \(\mathbb{Z}_n\), we have \(H \cong \mathbb{Z}_m\) and \(|\mathbb{Z}_n : H| = g\).
Denote the inverse of the map $k\mapsto (kd)_n$ from $\mathbb{Z}_m$ onto $H$ by $\varphi$. Then \(\varphi \colon H \to \mathbb{Z}_m\) is a group isomorphism sending \(\{0,d\} \mapsto \{0,1\}\). Since index stability is preserved under isomorphism \cite[Remark 3.16]{Hooshmand2020}, and since \(\{0,d\} \subseteq H \leq \mathbb{Z}_n\), by \cite[Theorem~1.11]{Hooshmand2026arXiv},
\[
    |\mathbb{Z}_n : \{0,d\}|^+ = g \cdot |H : \{0,d\}|^+=g|\mathbb{Z}_m : \{0,1\}|^+=g \left\lfloor \frac{m}{2} \right\rfloor= g \left\lfloor \frac{n}{2g} \right\rfloor
    =\left\lfloor \frac{n}{2} \right\rfloor_{g}.
\]
Analogously, for the lower index.\\
Therefore,  by translation and inverse invariance,
\[
|\mathbb{Z}_n:\{a,b\}|^\pm=|\mathbb{Z}_n:b+_n\{0,(a-b)_n\}|^\pm
=
|\mathbb{Z}_n:\{0,(a-b)_n\}|^\pm =|\mathbb{Z}_n:\{0,d\}|^\pm  .
\]
\end{proof}
\begin{corollary}[General formulas for subindices of 2-subsets of $\mathbb{Z}_n$]
For every pair of distinct elements $a,b$ of $\mathbb{Z}_n$, we have
\begin{equation}
|\mathbb{Z}_n : \{a, b\} |^+ = \gcd(n,a-b) \left\lfloor \frac{n}{2\gcd(n,a-b)} \right\rfloor, \;
|\mathbb{Z}_n : \{a, b\} |^- = \gcd(n,a-b) \left\lceil \frac{n}{3\gcd(n,a-b)} \right\rceil.
\end{equation}
\end{corollary}
Note that this is obtained from \eqref{eq:2subset-formula} since $\gcd(n,(a-b)_n)=\gcd(n,a-b)$.\\

\subsubsection*{Step 3.}
A 2-element set of the form $\{0,d\}$ is index stable if and only if $g \left\lfloor \frac{n}{2g}\right\rfloor = g\left\lceil \frac{n}{3g} \right\rceil$,
and if and only if $n/g\in  \{2,3,4,5,7\}$ (by Lemma 1.1). This means $\gcd(n,d)=g = \frac{n}{m}$ where $m\mid n$ and $m\in  \{2,3,4,5,7\}$.
Hence, we arrive at the following important result.
\begin{corollary}[Complete characterization of 2-stable subsets of \(\mathbb{Z}_n\)]
A \(2\)-subset \(\{a,b\}\) of \(\mathbb{Z}_n\) is index stable if and only if
\[
\frac{n}{\gcd(n,a-b)} \in \{2,3,4,5,7\}.
\]
\end{corollary}

\subsubsection*{Step 4.}
To count index stable 2-subsets, it is enough to count the subsets of the form $x+\{0,d\}$, in Step 1, with the mentioned ranges.
Since $\gcd(n,d) = \frac{n}{m}$ ($m\in \{2,3,4,5,7\}$), we may write $d = \frac{n}{m}\,k$ for some positive integer $k$, and then
\[
\gcd(n,d) = \frac{n}{m}\,\gcd(m,k),
\]
so the condition $\gcd(n,d)=\frac{n}{m}$ is equivalent to $\gcd(m,k)=1$. Now, consider the following two cases.

\begin{enumerate}[label=Case \arabic*.]

\item $n$ is odd. Note that $m$ must be odd, so $m\in\{3,5,7\}$.
The constraint $1 \leq d \leq \frac{n-1}{2}$ becomes $1 \leq k \leq \frac{m(n-1)}{2n} < \frac{m}{2}$,
so $k$ ranges over $\left\{1,\ldots,\frac{m-1}{2}\right\}$.
Since $\gcd(k,m)=\gcd(m-k,m)$, the $\varphi(m)$ integers in $\{1,\ldots,m-1\}$ coprime to $m$
pair up into $\frac{\varphi(m)}{2}$ pairs $\{k,\, m-k\}$ with exactly one element in
$\left\{1,\ldots,\frac{m-1}{2}\right\}$. Hence the number of valid $d$ is $\dfrac{\varphi(m)}{2}$.
Therefore,
\[
\#\{\text{index stable }2\text{-subsets of }\mathbb{Z}_n\}
=\sum_{\substack{m\in\{3,5,7\}\\ m\mid n}} n\,\frac{\varphi(m)}{2}
=\frac{n}{2}\left(2\,\mathbf{1}_{3\mid n}+4\,\mathbf{1}_{5\mid n}+6\,\mathbf{1}_{7\mid n}\right),
\]
where $\mathbf{1}_{q\mid n}$ denotes the indicator that $q\mid n$.

\item $n$ is even.
By Step 1, the $2$-element subsets of $\mathbb{Z}_n$ split into two families.

\smallskip
\noindent\emph{Family~I: subsets of the form $x+_n \{0,d\}$ with $1\leq d\leq \frac{n}{2}-1$ and
$0\leq x\leq n-1$.}
As before, $\gcd(n,d)=\frac{n}{m}$ requires $d=\frac{n}{m}k$ with $\gcd(k,m)=1$.
The constraint $1\leq d\leq \frac{n}{2}-1$ gives
\[
1\leq k\leq \frac{m(n/2-1)}{n}=\frac{m}{2}-\frac{m}{n}<\frac{m}{2}.
\]
Hence $m>2$, and we get $k=1$ when $m=4$, and
$k\in\left\{1,\ldots,\frac{m-1}{2}\right\}$ when $m$ is odd.

\begin{enumerate}[label=(\roman*)]
\item If $m$ is \emph{odd} (so $m\in\{3,5,7\}$ and $m\mid n$): the same pairing argument as
in Case~1 gives $\dfrac{\varphi(m)}{2}$ valid values of $k$, contributing
$n\cdot\dfrac{\varphi(m)}{2}$ index stable subsets from Family~I.
\item If $m=4$: the only value is $k=1$, and $\gcd(1,4)=1$ automatically, so this single
value of $k$ is valid; since $\varphi(4)/2=1$, this again contributes
$n\cdot\dfrac{\varphi(4)}{2}$ index stable subsets from Family~I.
\end{enumerate}

Hence Family~I contributes $n\cdot\dfrac{\varphi(m)}{2}$ index stable subsets for every
$m\in\{3,4,5,7\}$ dividing $n$, and (since $m>2$ was forced above) no subsets with $m=2$.

\smallskip
\noindent\emph{Family~II: subsets of the form $x+_n \{0,\frac{n}{2}\}$ with
$0\leq x\leq \frac{n}{2}-1$.}
Here $d=\frac{n}{2}$ is fixed, so $\gcd(n,d)=\gcd(n,\frac{n}{2})=\frac{n}{2}$,
which equals $\frac{n}{m}$ only for $m=2$. When $2\mid n$, Family~II consists of exactly
$\frac{n}{2}$ subsets, all with $\gcd(n,d)=\frac{n}{2}$; since $d=\frac n2$ was shown above to be
excluded from Family~I, these subsets are disjoint from Family~I. Thus Family~II contributes
$\dfrac n2=n\cdot\dfrac{\varphi(2)}{2}$ index stable subsets, supplying the entire $m=2$ term.
\end{enumerate}

\smallskip
Combining Family~I (for $m\in\{3,4,5,7\}$) with Family~II (for $m=2$), we obtain, for $n$ even,
\[
\#\{\text{index stable }2\text{-subsets of }\mathbb{Z}_n\}
=\sum_{\substack{m\in\{2,3,4,5,7\}\\ m\mid n}} n\,\frac{\varphi(m)}{2}.
\]

\smallskip
Summing over all $m\in\{2,3,4,5,7\}$ with $m\mid n$, we arrive at the following corollary that is valid
for both cases.
\begin{corollary}[An exact explicit formula for the number of 2-stable subsets of \(\mathbb{Z}_n\)]
\[
\#\{\text{index stable }2\text{-subsets of }\mathbb{Z}_n\}=
\sum_{\substack{m\in\{2,3,4,5,7\}\\ m\mid n}}
\frac{n\varphi(m)}{2}
=
\frac n2
\left(
\mathbf{1}_{2\mid n}
+
2\,\mathbf{1}_{3\mid n}
+
2\,\mathbf{1}_{4\mid n}
+
4\,\mathbf{1}_{5\mid n}
+
6\,\mathbf{1}_{7\mid n}
\right).
\]
\end{corollary}
Note that for odd $n$ the term $\mathbf{1}_{2\mid n}=\mathbf{1}_{4\mid n}=0$ vanishes automatically,
recovering Case~1. The formula is therefore valid uniformly in $n$.\\

For studying the density of index stable $2$-subsets, since $|P_2(\mathbb{Z}_n)|=\binom{n}{2}=\frac{n(n-1)}{2}$,
The \emph{density} of index stable $2$-subsets is therefore
\[
\iota\sigma_2(\mathbb{Z}_n):
=\frac{\displaystyle\sum_{\substack{m\in\{2,3,4,5,7\}\\m\mid n}}\varphi(m)}{n-1}\leq \frac{n-\varphi(n)-1}{n-1},
\]
where the inequality holds for all $n>7$.

In particular, if $p>7$ (resp. $p\leq 7$) is a prime, then $\iota\sigma_2(\mathbb{Z}_p)=0$ (resp. $\iota\sigma_2(\mathbb{Z}_p)=1$).

Therefore, the density of index unstable $2$-subsets is
\[
1-\iota\sigma_2(\mathbb{Z}_n)
=\frac{n-1-\displaystyle\sum_{\substack{m\in\{2,3,4,5,7\}\\m\mid n}}\varphi(m)}{n-1}\geq \frac{\varphi(n)}{n-1}; \quad n>7.
\]
Since the numerator of $\iota\sigma_2(\mathbb{Z}_n)$ is bounded above by
$\varphi(2)+\varphi(3)+\varphi(4)+\varphi(5)+\varphi(7)=15$,
we have $\iota\sigma_2(\mathbb{Z}_n)\leq\frac{15}{n-1}\to 0$ as $n\to\infty$, so
almost all $2$-subsets of $\mathbb{Z}_n$ are index unstable for large $n$.\\

Regarding the bounds mentioned above (for the number of index stable and unstable 2-subsets) it is interesting to state
the following result.
\begin{lemma}
The following conditions are equivalent:
\begin{enumerate}
\item $\#\{\text{index unstable }2\text{-subsets of }\mathbb{Z}_n\} = \dfrac{n\,\varphi(n)}{2}$ (i.e., it attains the lower bound),
\item $\#\{\text{index stable }2\text{-subsets of }\mathbb{Z}_n\} = \dfrac{n(n-\varphi(n)-1)}{2}$ (i.e., it attains the upper bound),
\item $n\notin\{2,3,4,5,7\}$ and every proper divisor $m\in(1,n)$ of $n$ belongs to $\{2,3,4,5,7\}$,
\item $n$ is prime with $n>7$, or $n\in\{6,8,9,10,14,15,21,25,35,49\}$.
\end{enumerate}
\end{lemma}

\begin{proof}
Conditions (1) and (2) are equivalent since their counts sum to $\binom{n}{2}$.
Substituting the counting formula into (1) and applying the Gauss identity
$\sum_{m\mid n}\varphi(m)=n$ gives the equivalent condition
\[
\sum_{\substack{m\in\{2,3,4,5,7\}\\m\mid n}}\varphi(m)
=n-1-\varphi(n)
=\sum_{\substack{m\mid n\\1<m<n}}\varphi(m).
\]
When $n\in\{2,3,4,5,7\}$, the term $m=n$ appears on the left but not the right,
so the equality fails. For $n\notin\{2,3,4,5,7\}$, the equality holds if and only
if every proper divisor $m\in(1,n)$ of $n$ belongs to $\{2,3,4,5,7\}$,
establishing (1)$\Leftrightarrow$(3).

For (3)$\Leftrightarrow$(4): if $n>7$ is prime, it holds vacuously.
For composite $n$ satisfying (3), every prime factor must lie in $\{2,3,5,7\}$ and
no proper divisor may exceed $7$. Prime powers $p^k$ satisfy this if and only if
$p^{k-1}\leq 7$, giving $n\in\{8,9,25,49\}$; products of two distinct primes
$p<q$ in $\{2,3,5,7\}$ require $pq=n$ (else $pq$ is a proper divisor outside
$\{2,3,4,5,7\}$), giving $n\in\{6,10,14,15,21,35\}$. All other composite $n$ fail (3).
\end{proof}
\begin{remark}
A related problem is to determine all $n$ such that
\[
\#\{\text{index-unstable }2\text{-subsets of }\mathbb{Z}_n\} = n\varphi(n).
\]
Equivalently,
\[
\sum_{\substack{m\in\{2,3,4,5,7\}\\m\mid n}} \varphi(m) = n-1-2\varphi(n).
\]
Since the left-hand side  is at most $15$, any solution satisfies $n-1-2\varphi(n)\le 15$.

Known solutions include
\[
\{2p : p>7\ \text{prime}\}\;\cup\;\{4p : p>7\ \text{prime}\}\;\cup\;\{40, 15561, \ldots\}.
\]
The determination of all solutions is left as a number theoretic problem.
\end{remark}

\subsubsection*{Step 5.}
In this step we determine all $n$ such that \(\mathbb{Z}_n\) has no index stable 2-subset
(i.e., $\iota\sigma_2(\mathbb{Z}_n)=0$). Note that
all 2-subsets of \(\mathbb{Z}_n\) are index stable (i.e., $\iota\sigma_2(\mathbb{Z}_n)=1$) if and only if $n\in \{2,3,4,5,7\}$.
\begin{theorem}[Complete characterization of the groups \(\mathbb{Z}_n\) with 2-stable density 0]
Let \(n > 7\) be a positive integer.
All \(2\)-subsets of \(\mathbb{Z}_n\) are index unstable if and only if
\(\gcd(n,210) = 1\) (i.e., every prime factor of \(n\) is greater than or equal to \(11\)).
\end{theorem}
\begin{proof} By the formula in Step 4, all \(2\)-element subsets of \(\mathbb{Z}_n\) are index unstable if and only if
none of \(2, 3, 5, 7\) divides \(n\).
Since \(210 = 2\cdot 3\cdot 5\cdot 7\), this is precisely the condition \(\gcd(n,210) = 1\).
\end{proof}
The next project should be an important basic part of the big project of
$k$-index stability of finite groups (see also the next section).\\

\textbf{Project 1.} Extend the study of index stability of 2-subsets for $k$-subsets  of
 $\mathbb{Z}_n$ where $2\leq k\leq \lfloor \frac{n}{3}\rfloor$ (for other values of $k$, it is $k$-index stable).\\
\textbf{Note.}  For $k=3$, let $A=\{a,b,c\}\subseteq\mathbb{Z}_n$. By translation invariance,
\[
|\mathbb{Z}_n:\{a,b,c\}|^\pm
=
|\mathbb{Z}_n:\{0,\,b-_n a,\,c-_n a\}|^\pm .
\]
Hence every $3$-subset is equivalent, with respect to its subindices, to a subset of the form
$\{0,d,e\}$. More generally, if $u$ is a unit in $(\mathbb{Z}_n,+_n,\cdot_n)$, then multiplication by $u$ is an
automorphism of $(\mathbb{Z}_n,+_n)$, so (see Remark~3.17 of \cite{Hooshmand2026arXiv})
\[
|\mathbb{Z}_n:\{0,d,e\}|^\pm
=
|\mathbb{Z}_n:\{0,u\cdot_n d,u\cdot_n e\}|^\pm .
\]
In particular, whenever $\gcd(d,n)=1$, taking $u=d^{-1}$ gives
\[
|\mathbb{Z}_n:\{0,d,e\}|^\pm
=
|\mathbb{Z}_n:\{0,1,d^{-1} \cdot_n e\}|^\pm .
\]
Thus every $3$-subset containing a unit difference has the same subindices as some
\emph{normalized form} $\{0,1,t\}$, $t\in\mathbb{Z}_n\setminus\{0,1\}$. Unlike the case of
$2$-subsets, there is no single normal form here: the subindices of $\{0,1,t\}$, $t\in\mathbb{Z}_n\setminus\{0, 1\}$, genuinely
depend on $t$ and not only on $n$, so stability of a $3$-subset is governed by an entire
family of normal forms rather than a single criterion. We end this section with the following question.

\textbf{Question 1.} Let $n>11$ be composite. Is it true that
\[
\#\{\text{index unstable $3$-subsets of }\mathbb{Z}_n\}\;\ge\;\frac{n\varphi(n)(n-2)}{6}?
\]

\section{Basic results for $k$-index stability of finite groups}
The following open problems are among the most important problems in the theory of subindices and subfactors of (finite) groups.
\\

\textbf{Problem II (\cite{Hooshmand2020})}. Let $k\geq 2$ be a given natural number.
\begin{enumerate}[label=(\alph*)]
\item Characterize all $n$ such that $\mathbb{Z}_n$  is $k$-index stable.
\item Characterize or classify all finite groups $G$ of order $m$ such that it is $k$-index
stable (resp. right $k$-index stable), where $m$ is a fixed integer with $k\leq \lfloor \frac{m}{2}\rfloor$ (resp. $k\leq \lfloor \frac{m}{3}\rfloor$).
\end{enumerate}

\textbf{Conjecture 4.2 (\cite{HooshmandYousefian}).} If $n > 11$, then $\mathbb{Z}_n$ is not $k$-index stable if and only if $2\leq k \leq \lfloor \frac{n}{3} \rfloor$.\\

{\bf New problem (\cite{HooshmandYousefian}).}
\begin{enumerate}[label=(\alph*)]
\item Determine all finite abelian (resp. non-abelian) groups that are not $k$-index stable for all  $2 \leq k \leq \lfloor \frac{|G|}{3}\rfloor$ (resp. $2 \leq k \leq \lfloor \frac{|G|}{2}\rfloor$). Also, determine all finite non-abelian groups that are not right $k$-index stable for all $2 \leq k \leq \lfloor \frac{|G|}{3}\rfloor$.
\item Characterize or classify all finite groups containing an absolutely strong right index \underline{unstable} subset.
\end{enumerate}

Indeed, all of them are about the status of $k$-index stability of finite groups, and especially $\mathbb{Z}_n$.
One sees that we have answered Problem II (a) and Conjecture 4.2 in this paper. Also,  New problem (a) is answered for $G=\mathbb{Z}_n$,
and  Project 1 (in the previous section) is a special part related to Problem II(b). Also, the following lemma solves
New Problem (b).

\begin{lemma}
A finite group \(G\) contains a subset that is absolutely strong right index unstable if and only if \(G\) is cyclic of order \(6\).
\end{lemma}

\begin{proof}
Suppose that $A\subseteq G$ is absolutely strong right index unstable, that is
\[
|G:A|^- =2,\;
|G:A|^+=\left\lfloor\frac{|G|}{2}\right\rfloor>2 .
\]

By the general bounds,
 we obtain
\[
\left\lfloor\frac{|G|}{|A|}\right\rfloor
\ge
\left\lfloor\frac{|G|}{2}\right\rfloor>2,
\]
hence $|A|\le 2$ and $|G|\geq 6$. As $A$ is not index stable, we cannot have $|A|=1$, so $|A|=2$.

After translating, write $A=\{1,g\}$ with $g\neq 1$. Thus
\[
A^{-1}A=\{1,g,g^{-1}\}.
\]
If $|A^{-1}A|=2$, then
\[
2=|G:A|^-\ge \left\lceil\frac{|G|}{2}\right\rceil,
\]
which contradicts $|G|\geq 6$. Therefore $|A^{-1}A|=3$, and hence
\[
2=|G:A|^-
\ge
\left\lceil\frac{|G|}{3}\right\rceil.
\]
Thus $|G|\le6$, and so $|G|=6$. Therefore $G$ is isomorphic to $C_6$, since $S_3$ is index stable and
\[
|C_6:\{0,1\}|^- =2,\;
|C_6:\{0,1\}|^+=3.
\]
\end{proof}
\textbf{Note.} The similar problem to the above lemma for infinite groups is still open, in the sense that $G$ contains a subset
$A$ such that $|G:A|^- =2$ and $|G:A|^+=|G|$.\\

Due to the above discussion, the only remaining open problems among those mentioned are "Characterization or classification of finite $k$-index stable groups"
and  "Determining all finite abelian (resp. non-abelian) groups that are not $k$-index stable for all  $2 \leq k \leq \lfloor \frac{|G|}{3}\rfloor$ (resp. $2 \leq k \leq \lfloor \frac{|G|}{2}\rfloor$)". For these, we also obtain the following important necessary condition by using the main results of this paper.
\begin{theorem}
Fix an integer $k\geq 4$. A necessary condition for a finite group $G$ to be right $k$-index stable (also $k$-index stable) is $2\le p\le 3k-1$,
for all prime divisors $p$ of $|G|$. Hence, all finite right $k$-index stable groups are among the groups of orders
\begin{equation}
n=\prod_{p\in \mathbb{P}\cap [2,3k-1]}p^{\alpha_p}\;\;\; , \;\;\; \alpha_p\ge 0
\end{equation}
For the case $k=2$ (resp. $k=3$), the order $n$ must be of the form $n=2^{\alpha_2}\cdot 3^{\alpha_3}\cdot 5^{\alpha_5}\cdot 7^{\alpha_7}$
(resp. $n=2^{\alpha_2}\cdot 3^{\alpha_3}\cdot 5^{\alpha_5}\cdot 7^{\alpha_7}\cdot 11^{\alpha_{11}}$).\\
\end{theorem}
\begin{proof}
Let  $n=|G|$, $p\geq k\geq 4$, $p>11$, and $p\mid n$. If $G$ is right $k$-index stable, then all subgroups of $G$ of sizes $\geq k$, including $C_p$,
are right $k$-index stable by Corollary 3.23 of \cite{Hooshmand2020}, and so $k>\lfloor \frac{p}{3} \rfloor$ (by Corollary 3.4 of \cite{Hooshmand2023}).
Therefore $p\le 3k-1$.
For  $p\leq 11$, the inequality $p\le 3k-1$ also holds, since $k\geq 4$.
\\
The last part (the case $k=2,3$) is also a consequence of Theorem 2.1 and Lemma 1.1, similar to the argument above.
\end{proof}
This completes our study of $k$-index stability in cyclic groups and provides a foundation for the classification or characterization of finite
$k$-index stable groups. It is worth noting that a finite group is (right/left) index stable if and only if it is (right/left) $k$-index stable for all $2\le k\le \lfloor |G|/3\rfloor$.
The complete characterization of index stable groups has been done in \cite{HooshmandYousefian, Hooshmand2026arXiv}: there are exactly 14 finite (right/left) index stable groups (up to isomorphism), and no infinite group is index stable.

\section*{Acknowledgments}
I would like to express my sincere gratitude to Professor Mikhail Kabenyuk for his valuable comments and insightful suggestions, which significantly improved the quality and presentation of this paper.
\bibliographystyle{amsplain}

\begin{thebibliography}{99}
\bibitem{Furstenberg1981}
  H.~Furstenberg, \textit{Recurrence in Ergodic Theory and Combinatorial Number Theory},
Princeton University Press, Princeton, 1981.

\bibitem{Hooshmand2013}
 M.~H.~Hooshmand,\textit{$b$-Parts and Finite $b$-Representation of Real Numbers},
Notes on Number Theory and Discrete Mathematics,
\textbf{19}(4)(2013), 4--15.

\bibitem{Hooshmand2020}
 M.~H.~Hooshmand,
  \textit{Index, sub-index and sub-factor of groups with interactions to number theory},
J. Alg. Appl., \textbf{19}(6)(2020), 1--23.

\bibitem{Hooshmand2021}
 M.~H.~Hooshmand, \textit{$f$-Representatives groups},
International Electronic Journal of Algebra,
\textbf{30}(2021), 66--77.

\bibitem{Hooshmand2023}
  M.~H.~Hooshmand,
  \textit{\it Subindices and subfactors of finite groups},
		Commun. Algebra,  \textbf{51}(6)(2023), 2644--2657.

\bibitem{HooshmandYousefian}
 M.~H.~Hooshmand and M.~M.~Yousefian Arani,
  \textit{Computational aspects of subindices and subfactors with characterization of finite index stable groups},
Int. J. Group Theory, \textbf{15}(3)(2026), 145--160.

\bibitem{Hooshmand2026arXiv}
  M.~H.~Hooshmand,
  \textit{\it Subindices and subfactors of infinite groups and numbers},
		arXiv:2604.08724 [math.GR], (2026).
\bibitem{Kabenyuk}
 M.~Kabenyuk,
\textit{Factors in finite groups and well-covered graphs},
	arXiv:2602.06770 [math.GR], (2026).

\bibitem{Lyaskovska2007}
N.~Lyaskovska,
\textit{Constructing subsets of a given packing index in Abelian groups},
Acta Univ. Carol., Math. Phys., \textbf{48} (2) (2007), 69--80.

\bibitem{Protasov2011}
  I.~V.~Protasov,
  \textit{Selective survey on subset combinatorics of groups},
  J. Math. Sci., \textbf{174} (4) (2011), 486--514.

\end{thebibliography}

\end{document}